\documentclass[10 pt]{amsart}

\usepackage{amssymb,amsmath,amsfonts,mathrsfs}

\newtheorem{theorem}{Theorem}[section]
\newtheorem{thm}{Theorem}[section]

\newtheorem{lem}[theorem]{Lemma}
\newtheorem{cor}[theorem]{Corollary}
\newtheorem{prop}[theorem]{Proposition}

\theoremstyle{definition}

\theoremstyle{remark}
\newtheorem{remark}[theorem]{Remark}
\newtheorem{rem}[theorem]{Remark}

\numberwithin{equation}{section}

\newcommand{\loc}{\operatorname{loc}}
\newcommand{\Dom}{\operatorname{Dom}}
\newcommand{\mcomp}{C_{c}^{\infty}(M)}
\newcommand{\ecomp}{C_{c}^{\infty}(E)}
\newcommand{\lloc}{L_{\loc}}
\newcommand{\Del}{\Delta}

\newcommand{\RE}{\operatorname{Re}}

\newcommand{\End}{\operatorname{End}}

\newcommand\RR{\mathbb{R}}

\author{Ognjen Milatovic}
\address{Department of Mathematics
and Statistics
\\ University of North Florida
\\ Jacksonville, FL
32224 \\ USA}
\email{omilatov@unf.edu}
\author{Hemanth Saratchandran}
\address{Beijing International Center for Mathematical Research
\\ No. 5 Yiheyuan Road, Haidian District
\\ Beijing 100871\\ China P.R.}
\email{hemanth.saratchandran@gmail.com}

\title{Inequalities and separation for covariant Schr\"odinger operators}

\keywords{covariant Schr\"odinger operator, Riemannian manifold, separation}

\subjclass{35P05,47B25, 58J50}
\begin{document}
\maketitle
\begin{abstract}
We consider a differential expression $L^{\nabla}_{V}=\nabla^{\dagger}\nabla+V$,
where $\nabla$ is  a metric covariant derivative on a Hermitian bundle $E$ over a
geodesically complete Riemannian manifold $(M,g)$ with metric $g$, and $V$  is a
linear self-adjoint bundle map on $E$. In the language of
Everitt and Giertz, the differential expression $L^{\nabla}_{V}$ is said
to be separated in $L^p(E)$ if for all $u\in L^p(E)$ such that
$L^{\nabla}_{V}u\in L^p(E)$, we have $Vu\in L^p(E)$. We give sufficient
conditions for $L^{\nabla}_{V}$ to be separated in $L^2(E)$.
We then study the problem of separation of
$L^{\nabla}_{V}$ in the more general $L^p$-spaces, and give sufficient conditions for $L^{\nabla}_{V}$ to be
separated in $L^p(E)$, when $1<p<\infty$.

\end{abstract}

% main text

\section{Introduction}\label{S:intro-1}

The study of the separation property for Schr\"odinger operators on $\RR^n$ was initiated through
the work of Everitt and Giertz in \cite{Everitt-Giertz77}. We recall that the expression $-\Delta + V$ in $L^p(\RR^n)$
is \textit{separated} if the following property is satisfied: For all $u \in L^p(\RR^n)$ such that
$(-\Delta + V)u \in L^p(\RR^n)$, we have that $-\Delta u \in L^p(\RR^n)$ and $Vu \in L^p(\RR^n)$.
After the work of Everitt and Giertz, various authors took up the study of separation problems for
(second and higher order) differential operators; see~\cite{Boimatov88,Boimatov97,hdn-12,Okazawa-84} and references therein. The paper \cite{Milatovic06-separation} then studied the separation property of the operator
$\Del_M + v$ on $L^2(M)$, where $M$ is a non-compact Riemannian manifold,
$\Del_M$ is the scalar Laplacian, and $v \in C^1(M)$.
The separation problem for the differential expression $\Del_M + v$
in $L^p(M)$ was first considered, in the bounded geometry setting,
in \cite{Milatovic13-separation}. The work in \cite{Milatovic18-separation} gives another proof of the main theorem in \cite{Milatovic13-separation}, crucially without any bounded geometry hypothesis. For a study of separation in the context of
a perturbation of the (magnetic) Bi-Laplacian on $L^2(M)$, see the papers \cite{AAR,Milatovic18-separation}. A closer look at the works mentioned in this paragraph reveals that the separation property is linked to the self-adjointness in $L^2$ (or $m$-accretivity in $L^p$) of the underlying operator. In the context of a Riemannian manifold $M$, the latter problem has been studied quite a bit over the past two decades. For recent references see, for instance, the papers~\cite{bs-16,GK,GP,Pran-Ser-Riz-16} and chapter XI in~\cite{Guneysu-2016}.

In this article we consider the separation problem for the differential expression
$\nabla^{\dagger}\nabla+V$, where $\nabla$ is a metric covariant derivative on a Hermitian bundle $E$ over a Riemannian
manifold $M$, $\nabla^{\dagger}$ its formal adjoint, and $V$ is a self-adjoint endomorphism of $E$. We start with the
separation problem on $L^2(E)$, obtaining a result (see Theorem \ref{T:main})
that can be seen as an extension of the work carried out in
\cite{Milatovic06-separation}.  The condition~(\ref{E:assumption-2}) on the endomorphism $V$ that guarantees the separation property is analogous to the one in the scalar case. We then move on to consider the separation problem of the above operator in
$L^p(E)$, $1<p<\infty$, obtaining a result (see Theorem \ref{T:main-2}) that generalizes the work~\cite{Milatovic13-separation}.
We do this (see Proposition~\ref{L:c-e} and Corollary~\ref{C:domination} below for precise statements) by exploiting a coercive estimate for $\Del_M + v$ in $L^p(M)$ from~\cite{Milatovic13-separation} alongside the following property from~\cite{Guneysu-12}: if $V\geq v$, where $v\geq 0$ is a real-valued function on $M$, then the $L^p$-semigroup corresponding to $\nabla^{\dagger}\nabla+V$ is dominated by the $L^p$-semigroup corresponding to $\Delta_M+v$. In the case $p \neq 2$, we  assumed, in addition to geodesic completeness of $M$, that the Ricci curvature of $M$ is bounded from below by a constant. One reason is that, as far as we know, the only available proof of the coercive estimate~(\ref{M:Del_bound}) in the case $p\neq 2$ uses the Kato inequality approach, which leads one to apply, in the
language of section XIII.5 of~\cite{Guneysu-2016}, the $L^{p}$-positivity preservation property of $M$. The latter property,
whose proof is based on the construction of a sequence of Laplacian cut-off functions (see section III.1 in~\cite{Guneysu-2016}
for details), is known to hold under the aforementioned assumption on the Ricci curvature. Actually, as
shown in~\cite{B-S-16}, this hypothesis on Ricci curvature can be further weakened to assume boundedness below
by a (possibly unbounded) non-positive function depending on the distance from a reference point.
We mention in passing that the $L^{p}$-positivity preservation property of $M$ is related to the
so-called BMS-conjecture, the details of which are explained in~\cite{bms} and~\cite{Guneysu-Ulmer}.
Another reason for the hypothesis on Ricci curvature is that this assumption is used (see section~\ref{SS:op-h} below for details) for the $m$-accretivity of $\nabla^{\dagger}\nabla+V$  in $L^p(E)$ in the case $p\geq 3$.

Lastly, we should point out that although the separation property for $\Del_M + v$ in $L^p(M)$, with $v\geq 0$, was obtained in~\cite{Milatovic18-separation} under the geodesic completeness assumption on $M$ only, it was done so without explicitly establishing~(\ref{M:Del_bound}). Instead, assuming~(\ref{E:assumption-2}) with $v$ replaced by the Yosida approximation $v_{\varepsilon}:=v(1+\varepsilon v)^{-1}$, $\varepsilon>0$, and with a certain condition on the constant $\gamma$, the work~\cite{Milatovic18-separation} establishes an estimate involving the operator $\Del_{M}$ and the multiplication operator by $v_{\varepsilon}$. Using the abstract framework of~\cite{Okazawa-84}, one concludes the $m$-accretivity of the (operator) sum of ``maximal" operators corresponding to $\Del_{M}$  and $v$, which, due to the fact (see section~\ref{SS:op-h} below) that the ``maximal" operator corresponding to $\Del_M + v$ is $m$-accretive in $L^p(M)$, $1<p<\infty$, leads to the separation property. The approach from~\cite{Milatovic18-separation} does not seem to carry to covariant Schr\"odinger operators.

\section{Main Results}\label{S:main}

\subsection{The setting}\label{SS:setting}
Let $M$ be a smooth connected Riemannian manifold without boundary, with metric $g$, and with Riemannian volume element $d\mu$.
Let $E$ be a vector bundle over $M$ with Hermitian structure $\langle\cdot, \cdot\rangle_{x}$ and the corresponding norms $|\cdot|_{x}$ on fibers $E_{x}$. Throughout the paper, the symbols $C^{\infty}(E)$ and $\ecomp$  denote smooth sections of $E$ and smooth compactly supported sections of $E$, respectively. The notation  $L^p(E)$, $1\leq p<\infty$, indicates the space of $p$-integrable sections of $E$ with the norm
\[
\|u\|_{p}:=\int_{M}|u(x)|^p\,d\mu.
\]
In the special case $p=2$, we have a Hilbert space $L^2(E)$ and we use  $(\cdot,\cdot)$ to denote the corresponding inner product. For local Sobolev spaces of sections we use the notation $W^{k,p}_{\loc}(E)$, with $k$ and $p$ indicating the highest order of derivatives and the corresponding $L^{p}$-space, respectively. For $k=0$ we use the simpler notation $L^{p}_{\loc}(E)$. In the case $E=M\times\mathbb{C}$, we denote the corresponding function spaces by $C^{\infty}(M)$, $\mcomp$, $L^p(M)$, $W^{k,p}_{\loc}(M)$, and $L^{p}_{\loc}(M)$.
%For smooth $1$-forms we use the symbol $\Omega^1(M)$, and for compactly supported ones we use $\Omega_{c}^1(M)$. The tangent and cotangent bundle of $M$ are indicated by $TM$ and $T^*M$ respectively.

In the remainder of the paper, $\nabla\colon C^{\infty}(E)\to C^{\infty}(T^*M\otimes E)$ stands for a smooth metric covariant derivative on $E$, and $\nabla^{\dagger}\colon C^{\infty}(T^*M\otimes E)\to C^{\infty}(E)$  indicates the formal adjoint of $\nabla$ with respect to $(\cdot,\cdot)$. The covariant derivative $\nabla$ on $E$ induces the covariant derivative $\nabla^{\End}$ on the bundle of endomorphisms $\End E$, making $\nabla^{\End} V$ a section of the bundle $T^*M\otimes (\End E)$.

We study a covariant Schr\"odinger differential expression
\begin{equation}\label{E:expression-L}
L^{\nabla}_{V}:=\nabla^{\dagger}\nabla+V,
\end{equation}
where $V$ is a linear self-adjoint bundle map $V\in \lloc^{\infty}(\End E)$.
To help us describe the separation property,  we define
\begin{equation}\label{E:D-p}
\textrm{D}^{\nabla}_p:=\{u\in L^p(E)\colon L^{\nabla}_{V}u\in L^p(E)\}, \qquad 1<p<\infty,
\end{equation}
where $L^{\nabla}_{V}u$ is understood in the sense of distributions. In the case of a real valued function $v\in \lloc^{\infty}(M)$, trivial bundle $E=M\times\mathbb{C}$ and $\nabla=d$, where $d$ is the standard differential, we will use the notations $L^{d}_{v}:=\Delta_{M}+v$, with $\Del_{M}:=d^{\dagger} d$ indicating the scalar Laplacian.

In general, it is not true that for all $u\in\textrm{D}^{\nabla}_p$
we have $\nabla^{\dagger}\nabla u\in L^p(E)$ and $Vu\in L^p(E)$ separately. Using the language of Everitt and Giertz (see\cite{Everitt-Giertz77}), we say that the differential
expression $L^{\nabla}_{V}=\nabla^{\dagger}\nabla+V$ is \emph{separated} in $L^p(E)$ when
the following statement holds true: for all $u\in\textrm{D}^{\nabla}_p$ we have $Vu\in L^p(E)$.

\subsection{Statements of the Results} Our first result concerns the separation property for $L^{\nabla}_{V}$ in $L^2(E)$. Before giving its exact statement, we describe the assumptions on $V$.

\noindent\textbf{Assumption (A1)} Assume that

\begin{enumerate}

\item [(i)] $V\in C^1(\End E)$ and $V(x)\geq 0$, for all $x\in M$, where the inequality is
understood in the sense of linear operators $E_{x}\to E_{x}$;

\item[(ii)] $V$ satisfies the  inequality

\begin{equation}\label{E:assumption-2}
|(\nabla^{\End} V)(x)|\leq \beta (\underbar{V}(x))^{3/2},\qquad \textrm{for all }x\in M,
\end{equation}
where $0\leq \beta<1$ is a constant, $|\cdot|$ is the norm of a linear operator $E_{x}\to (T^*M\otimes E)_{x}$, and
$\underbar{V}\colon M\to \mathbb{R}$ is defined by
\[
\underbar{V}(x)=\min(\sigma(V(x))),
\]
where $\sigma(V(x))$ is the spectrum of the operator $V(x)\colon E_{x}\to E_{x}$.
\end{enumerate}

We are ready to state the first result.

\begin{thm}\label{T:main} Assume that $(M,g)$ is a smooth geodesically complete
connected Riemannian manifold without boundary. Let $E$ be a Hermitian vector bundle over $M$ with a metric covariant derivative $\nabla$. Assume that $V$ satisfies the assumption (A1). Then
\begin{equation}\label{E:to-prove}
\|\nabla^{\dagger}\nabla u\|_{2}+\|Vu\|_{2}\leq C(\|L_{V}u\|_{2}+\|u\|_{2}),
\end{equation}
for all $u\in \textrm{D}^{\nabla}_2$ ,where $C\geq 0$ is a constant (independent of $u$). In particular, $L^{\nabla}_{V}$ is separated in $L^2(E)$.
\end{thm}

The second result concerns the separation for $L^{\nabla}_{V}$ in $L^p(E)$.

\begin{thm}\label{T:main-2} Let $1<p<\infty$. Assume that $(M,g)$ is a
geodesically complete connected Riemannian manifold without boundary. In the case $p\neq 2$, assume additionally that the Ricci curvature of $M$ is bounded from below by a constant. Furthermore, assume that
there exists a function $0\leq v\in C^1(M)$ such that
\begin{equation}\label{E:assumption-1}
v(x)I\leq V(x)\leq \delta v(x)I
\end{equation}
and
\begin{equation}\label{E:assumption-2}
|dv(x)|\leq \gamma v^{3/2}(x),\qquad \textrm{for all }x\in M,
\end{equation}
where $\delta\geq 1$ and $0\leq \gamma<2$ are constants, and $I\colon E_{x}\to E_{x}$ is the identity operator.
Then, the differential
expression $L^{\nabla}_{V}$ is separated in $L^p(E)$.
\end{thm}

\begin{rem}\label{R:rem-infty} \emph{From~(\ref{E:assumption-1}) it follows that
$0\leq V\in\lloc^{\infty}(\End E)$.}
\end{rem}

\section{Preliminaries on Operators}\label{SS:op-h} We start by briefly recalling some abstract terminology concerning $m$-accretive operators on Banach spaces.
A linear operator $T$ on a Banach space $\mathscr{B}$ is called \emph{accretive}
if
\[
\|(\xi+T)u\|_{\mathscr{B}}\geq \xi\|u\|_{\mathscr{B}},
\]
for all $\xi>0$ and all $u\in\Dom(T)$. By Proposition II.3.14 in~\cite{engel-nagel}, a densely defined accretive operator $T$ is closable and its closure $T^{\sim}$ is also accretive. A (densely defined) operator $T$ on $\mathscr{B}$ is called \emph{$m$-accretive} if it is accretive and $\xi+T$ is surjective for all $\xi>0$. A (densely defined)  operator $T$ on $\mathscr{B}$ is called \emph{essentially $m$-accretive} if it is accretive and $T^{\sim}$ is $m$-accretive. As the proof of our first result uses the notion of self-adjointness, we recall a link between $m$-accretivity and self-adjointness of operators on Hilbert spaces:  $T$ is a self-adjoint and non-negative operator if and only if $T$ is symmetric, closed, and $m$-accretive; see Problem V.3.32 in~\cite{Kato80}.

We now describe some known results on the (essential) $m$-accretivity of operators in $L^p(E)$ used in this paper. With $L^{\nabla}_{V}$ and $L^{d}_{v}$ as in section~\ref{SS:setting} and with $0\leq V\in\lloc^{\infty}(\End E)$ and $0\leq v\in\lloc^{\infty}(M)$,
we define an operator $H^{\nabla}_{p,V}$ as $H^{\nabla}_{p,V} u:=L^{\nabla}_{V}u$ with the domain $\textrm{D}^{\nabla}_{p}$ as in~(\ref{E:D-p}) and an operator
$H^{d}_{p,v}$ as $H^{d}_{p,v} u:=L^{d}_{v}u$ for all $u\in \textrm{D}^{d}_{p}$, where
\[
\textrm{D}^{d}_{p}:=\{u\in L^p(M)\colon L^{d}_{v}u \in  L^p(M)\}.
\]
For a geodesically complete manifold $M$ it is known that $(L^{d}_{v}|_{\mcomp})^{\sim}$ in $L^p(M)$, $1<p<\infty$, is $m$-accretive and it coincides with $H^{d}_{p,v}$. Moreover, under the same assumption on $M$ and for $1<p<3$, the operator $(L^{\nabla}_{V}|_{\ecomp})^{\sim}$ in $L^p(E)$, is $m$-accretive and it coincides with $H^{\nabla}_{p,V}$. Both of these statements are proven in~\cite{Strichartz-83} for $V=0$ and $v=0$, but the arguments there work for $0\leq V\in\lloc^{\infty}$ and $0\leq v\in\lloc^{\infty}$ without any change, as the non-negativity assumption makes $V$ and $v$ ``disappear" from the inequalities. It turns out that the $m$-accretivity result holds for $(L^{\nabla}_{V}|_{\ecomp})^{\sim}$ in $L^{p}(E)$ in the case $p\geq 3$ as well if, in addition to geodesic completeness, we assume that the Ricci curvature of $M$ is bounded from below by a constant. The latter statement was proven for manifolds of bounded geometry in Theorem 1.3 of~\cite{Mi-2010}, and it was observed in~\cite{GP-2017} that the statement holds if we just assume that $M$ is geodesically complete and with Ricci curvature bounded from below by a constant. For the explanation of why $(L^{\nabla}_{V}|_{\ecomp})^{\sim}$ coincides with $H^{\nabla}_{p,V}$, we again point the reader to~\cite{GP-2017}. As indicated above, in the case $p=2$, the term ``$m$-accretivity" in the above statements has the same meaning as the term ``self-adjointness."

\section{Proof of Theorem \ref{T:main}}  Working in the $L^2$-context only, we find it convenient to indicate by $\|\cdot\|$ and $(\cdot,\cdot)$ the norm and the inner product in the spaces $L^2(E)$ and $L^2(T^*M\otimes E)$.
In subsequent discussion, we adapt the approach from~\cite{Boimatov88, Everitt-Giertz77} to our setting.

\begin{lem}\label{L:lemma-1} Under the hypotheses of Theorem~\ref{T:main}, the following inequalities hold for all $u\in\ecomp$:
\begin{equation}\label{E:inequality-1}
\|\nabla^{\dagger}\nabla u\|+\|Vu\|\leq \widetilde{C}\|L_{V}u\|
\end{equation}
and
\begin{equation}\label{E:inequality-2}
\|V^{1/2}\nabla u\|\leq \widetilde{C}\|L_{V}u\|,
\end{equation}
where $L^{\nabla}_{V}$ is as in~(\ref{E:expression-L}), the notation $V^{1/2}$ means square root of the operator $V(x)\colon E_{x}\to E_{x}$, and $\widetilde{C}$ is a constant depending on $n=\dim M$, $m=\dim E_x$, and $\beta$.
\end{lem}

\begin{proof} %We will first prove that the following equality holds for all $\nu>0$ and $u\in\ecomp$:
%\begin{align}\label{E:equality}
%&\|L^{\nabla}_{V}u\|^2  =  \|Vu\|^2  +  \nu\|\nabla^{\dagger}\nabla u\|^2  +
%(1+\nu)\RE(Vu, \nabla^{\dagger}\nabla u) \nonumber\\
%&+ (1-\nu)(\nabla^{\dagger}\nabla u,L^{\nabla}_{V}u).
%\end{align}
By the definition of $L^{\nabla}_{V}$, for all $\nu>0$ and all
$u\in\ecomp$ we have
\begin{align}\label{E:equality}
&\|L^{\nabla}_{V}u\|^2  =  \|Vu\|^2+ \|\nabla^{\dagger}\nabla u\|^2+2\RE(\nabla^{\dagger}\nabla u,Vu)  \nonumber\\
&= \|Vu\|^2+
\nu\|\nabla^{\dagger}\nabla u\|^2+(1-\nu)\|\nabla^{\dagger}\nabla u\|^2+
+2\RE(\nabla^{\dagger}\nabla u,Vu)  \nonumber\\
&=  \|Vu\|^2+
\nu\|\nabla^{\dagger}\nabla u\|^2
+(1-\nu)\RE(\nabla^{\dagger}\nabla u,L^{\nabla}_{V}u-Vu)+2\RE(\nabla^{\dagger}\nabla u,Vu) \nonumber\\
&= \|Vu\|^2+\nu\|\nabla^{\dagger}\nabla u\|^2+(1-\nu)\RE(\nabla^{\dagger}\nabla u,L^{\nabla}_{V}u)+(1+\nu)\RE(\nabla^{\dagger}\nabla u,Vu),
\end{align}
Using integration by parts and the ``product
rule"
\[
\nabla(Vu)= (\nabla^{\End} V)u+V\nabla u,
\]
for all $u\in\ecomp$ we have
\begin{align}\label{E:long}
&\RE (\nabla^{\dagger}\nabla u, Vu)=  \RE (\nabla u,\nabla (Vu))  = \RE (\nabla u, (\nabla^{\End} V)u+V\nabla u)\nonumber\\
&=   \RE (\nabla u,(\nabla^{\End} V)u)+(\nabla u,V\nabla u)=   (\RE Z) + W,
\end{align}
where
\[
Z:=(\nabla u, (\nabla^{\End} V)u)
\]
and
\begin{equation}\nonumber
W:=(\nabla u,V\nabla u)=(V^{1/2}\nabla u,V^{1/2}\nabla u).
\end{equation}
From~(\ref{E:long}) we get
\begin{align}\label{E:long-1+nu}
&(1+\nu)\RE(\nabla^{\dagger}\nabla u,Vu)=(1+\nu)\RE Z+(1+\nu)W \nonumber\\
&\geq  -(1+\nu)|Z|+(1+\nu)W.
\end{align}
Using~(\ref{E:assumption-2}) and
\begin{equation}\label{E:helping}
2ab\leq ka^2+k^{-1}b^2,
\end{equation}
where $a$, $b$ and $k$ are positive real numbers, we obtain
\begin{align}\label{E:long-1}
&|Z| \leq  \ (\beta+1)\int_{M}|{\underbar{V}}^{1/2}
\nabla u|_{(T^*M\otimes E)_{x}}|\underbar{V}u|_{E_{x}}\,d\mu  \nonumber\\
&\leq  \frac{\nu\delta}{2}\|V^{1/2}\nabla u\|^2+\frac{(\beta+1)^2}{2\nu\delta}\|Vu\|^2,
\end{align}
for all $\delta>0$. Using~(\ref{E:helping}) again, we get
\begin{align}\label{E:long-2}
&|\RE (\nabla^{\dagger}\nabla u, L^{\nabla}_{V}u)|  \leq  |(\nabla^{\dagger}\nabla u, L^{\nabla}_{V}u)| \leq  \frac{\alpha}{2}\|\nabla^{\dagger}\nabla u\|^2+\frac{1}{2\alpha}\|L^{\nabla}_{V}u\|^2,
\end{align}
for all $\alpha>0$. Combining~(\ref{E:equality}),~(\ref{E:long-1+nu}),~(\ref{E:long-1})
and~(\ref{E:long-2}) we obtain
\begin{align}\nonumber
&\|L^{\nabla}_{V}u\|^2 \ \geq \|Vu\|^2 + \nu\|\nabla^{\dagger}\nabla u\|^2
-\frac{(1+\nu)\nu\delta}{2}\|V^{1/2}\nabla u\|^2-\frac{(1+\nu)(\beta+1)^2}{2\nu\delta}\|Vu\|^2\nonumber\\
&+(1+\nu)\|V^{1/2}\nabla u\|^2-\frac{|1-\nu|\alpha}{2}\|\nabla^{\dagger}\nabla u\|^2
-\frac{|1-\nu|}{2\alpha}\|L^{\nabla}_{V}u\|^2,\nonumber
\end{align}
which upon rearranging leads to
\begin{align}\nonumber
&\left(1+\frac{|1-\nu|}{2\alpha}\right)\|L^{\nabla}_{V}u\|^2 \geq
\left(1-\frac{(1+\nu)(\beta+1)^2}{2\nu\delta}\right)\|Vu\|^2 \nonumber\\
&+\left(\nu-\frac{|1-\nu|\alpha}{2}\right)\|\nabla^{\dagger}\nabla u\|^2
+
\left((1+\nu)-\frac{(1+\nu)\nu\delta}{2}\right)\|V^{1/2}\nabla u\|^2. \nonumber
\end{align}

Finally, we observe that~(\ref{E:inequality-1}) and~(\ref{E:inequality-2}) will
follow from the last inequality if
\begin{equation}\label{E:constants}
|1-\nu|<\frac{2\nu}{\alpha}, \qquad \nu\delta<2,
\qquad\textrm{and}\qquad (1+\nu)(\beta+1)^2<2\nu\delta.
\end{equation}
Since, by hypothesis, $0\leq\beta<1$, there exist numbers $\nu>0$,
$\alpha>0$ and $\delta>0$ such that the
inequalities~(\ref{E:constants}) hold.
\end{proof}

\bigskip

\noindent{\textbf{Continuation of the Proof of Theorem~\ref{T:main}}
As indicated in section~\ref{SS:op-h}, the operator $L^{\nabla}_{V}|_{\ecomp}$ is essentially self-adjoint and $(L^{\nabla}_{V}|_{\ecomp})^{\sim}=H^{\nabla}_{2,V}$.
We will show that~(\ref{E:inequality-1}) and~(\ref{E:inequality-2}) hold for
all $u\in \textrm{D}^{\nabla}_2=\Dom(H^{\nabla}_{2,V})$, from which~(\ref{E:to-prove}) follows directly.

As $H^{\nabla}_{2,V}$ is a closed operator, there exists a sequence $\{u_k\}$ in
$\ecomp$ such that $u_k\to u$ and $L^{\nabla}_{V}u_k\to H^{\nabla}_{2,V}$ in $L^2(E)$. By Lemma~\ref{L:lemma-1} the sequence $\{u_k\}$
satisfies~(\ref{E:inequality-1}) and~(\ref{E:inequality-2}); hence, $\{Vu_k\}$, $\{\nabla^{\dagger}\nabla u_k\}$, and $\{V^{1/2}\nabla u_k\}$ are Cauchy sequences in the appropriate $L^2$-space (corresponding to $E$ or $T^*M\otimes E$). Furthermore, $\{\nabla u_k\}$ is a Cauchy sequence in $L^2(T^*M\otimes E)$ because
\[
\|\nabla u_k\|^2=(\nabla u_k, \nabla u_k)=(\nabla^{\dagger}\nabla u_k, u_k)\leq \|\nabla^{\dagger}\nabla u_k\|\|u_k\|.
\]
It remains to show that $Vu_k\to Vu$, $V^{1/2}\nabla u_k\to V^{1/2}\nabla u$, and $\nabla u_k\to \nabla u$ in the appropriate $L^2$-space.
As the proofs of these three convergence relations follow the same pattern, we will only show the details for the third one. We start by observing that from the essential self-adjointness of $\nabla^{\dagger}\nabla|_{\ecomp}$ we get $\nabla^{\dagger}\nabla u_k\to \nabla^{\dagger}\nabla u$ in $L^2(E)$. Since $\{\nabla u_k\}$ is a Cauchy sequence in $L^2(T^*M\otimes E)$, it follows that $\nabla u_k$ converges to some element $\omega\in
L^2(T^*M\otimes E)$. Then, for all $\psi\in C_{c}^{\infty}(T^*M\otimes E)$ we have
\begin{equation}\nonumber
0 = (\nabla u_k,\psi)-(u_k,\nabla^{\dagger}\psi)\to (\omega,\psi)- (u,\nabla^{\dagger}\psi)=(\omega,\psi)-(\nabla u,\psi),
\end{equation}
where in the second equality we used integration by parts (see, for instance, Lemma 8.8 in~\cite{bms}), which is applicable because elliptic regularity tells us that $\Dom(H^{\nabla}_{2,V})\subset W^{2,2}_{\loc}(E)$.

With the three convergence relations at our disposal, taking the limit as $k\to\infty$ in all terms
in~(\ref{E:inequality-1}) and~(\ref{E:inequality-2}) (with $u$ replaced
by $u_k$) shows that~(\ref{E:inequality-1}) and~(\ref{E:inequality-2})
hold for all $u\in \textrm{D}^{\nabla}_2=\Dom(H^{\nabla}_{2,V})$. $\hfill\square$

\section{Proof of Theorem \ref{T:main-2}}

\subsection{Semigroup Representation Formula.}\label{S:semi-group}
Assuming that $M$ is geodesically complete (with Ricci curvature bounded from below by a constant in the case $p\geq 3$) and $0\leq V\in\lloc^{\infty}(\End E)$, the operator $H^{\nabla}_{p,V}$ is $m$-accretive (see section~\ref{SS:op-h}), and its negative, $-H^{\nabla}_{p,V}$, generates a strongly continuous contraction semigroup $S_t$ on $L^p(E)$, $1<p<\infty$; see abstract Theorem II.3.15 in~\cite{engel-nagel}.

Before stating a crucial proposition for the proof of Theorem~\ref{T:main-2}, we describe a probabilistic setting. In the subsequent discussion, we assume that the underlying filtered probability space $(\Omega,\mathscr{F},\mathscr{F}_{*},\mathbb{P})$, where $\mathscr{F}_{*}$ is right-continuous and the pair $(\mathbb{P},\mathscr{F}_{t})$ is complete in measure theoretic sense for all $t\geq 0$, carries a Brownian motion $W$ on $\mathbb{R}^{l}$ sped up by 2, that is,  $d[W_{t}^{j},W_{t}^{k}]=2\delta_{jk}\,dt$, where $\delta_{jk}$ is the Kronecker delta and $l\geq n=\dim M$ is sufficiently large.  We will also assume $\mathscr{F}_{*}=\mathscr{F}_{*}(W)$.  Let $B_{t}(x)\colon \Omega \times [0,\zeta(x))\to M$ be a Brownian motion starting at $x\in M$ with lifetime $\zeta(x)$. It is well known that this process can be constructed as the maximally defined solution of the Stratonovich equation
\[
dB_{t}(x)=\sum_{j=1}^{l}A_{j}(B_{t}(x))\underbar{d}W_{t}^{j},\quad B_{0}(x)=x,
\]
where $A_{j}$ are smooth vector fields on $M$ such that $\displaystyle\sum_{j=1}^{l}A^2_{j}=\Delta_{M}$.
\begin{remark} In the setting of Theorem~\ref{T:main-2}, for $p\neq 2$ our assumptions $M$ imply that $M$ is stochastically complete; hence, in this case we have $\zeta(x)=\infty$.
\end{remark}
In the sequel,  $\slash\slash_{t}^{x}\colon E_{x}\to  E_{B_{t}(x)}$ stands for the stochastic parallel transport corresponding to the covariant derivative $\nabla$ on $E$. Additionally, the symbol $\mathscr{V}^{x}_{t}$ stands for the $\End E_{x}$-valued process (with lifetime $\zeta(x)$) defined as the unique pathwise solution to
\[
d\mathscr{V}^{x}_{t}=-\mathscr{V}^{x}_{t}(\slash\slash^{t,-1}_{x}V(B_{t}(x)))\slash\slash_{t}^{x})\,dt, \qquad \mathscr{V}^{x}_{0}=I,
\]
where $\slash\slash^{t,-1}_{x}$ is the inverse of $\slash\slash_{t}^{x}$ and $I$ is the identity endomorphism.

We now state the proposition, which in the $p=2$ context is a special case of Theorem 1.3 in~\cite{Guneysu-10}. For an extension to possibly negative $V$, in the case $p=2$, see Theorem 2.11 in~\cite{Guneysu-12}. The proof of Theorem 1.3 in~\cite{Guneysu-10} is almost entirely applicable to the proposition below. Thus, we will only explain those parts in need of small changes to accommodate the general $1<p<\infty$.

\begin{prop}\label{P:FK-rep}  Let $1<p<\infty$, let $M$ be a (smooth) connected Riemannian manifold without boundary, and let $E$, $\nabla$ be as in Theorem~\ref{T:main-2}. Assume that $M$ is geodesically complete. In the case $p\geq 3$, assume additionally that the Ricci curvature of $M$ is bounded from below by a constant.  Assume that $V\in\lloc^{\infty}(\End E)$ satisfies the inequality $V(x)\geq 0$, for all $x\in M$. Let $S_t$ be the semigroup defined in section~\ref{S:semi-group}. Then, we have the representation
\begin{equation}\label{E:F-K}
(S_{t}f)(x)=\mathbb{E}\left[\mathscr{V}^{x}_{t}\slash\slash^{t,-1}_{x}f(B_{t}(x))1_{\{t<\zeta(x)\}} \right],
\end{equation}
for all $f\in L^p(E)$.
\end{prop}
\begin{proof} We first assume that $0\leq V\in C(\End E)\cap L^{\infty}(\End E)$.
Denoting by $L^{0}(E)$ Borel measurable sections, define a family of operators $Q_{t}\colon L^{p}(E)\to L^{0}(E)$, $t\geq 0$, as
\[
(Q_{t}h)(x):=\mathbb{E}\left[\mathscr{V}^{x}_{t}\slash\slash^{t,-1}_{x}h(B_{t}(x))1_{\{t<\zeta(x)\}} \right].
\]
We will show that $Q_{t}$ are bounded operators $L^{p}(E)\to L^{p}(E)$. Using H\"{o}lder's inequality, on letting $q$ be the dual exponent to $p$, we can estimate
\begin{align*}
\|Q_th\|_{p}^p &\leq e^{pt\|V\|_{\infty}}\int_{M}\mathbb{E}[|h(B_t(x))|_{B_t(x)}]^p\,d\mu(x) \\
           &= e^{pt\|V\|_{\infty}}\int_M\bigg{(} \int_M|h(y)|_y\rho_t(x,y)\,d\mu(y)\bigg{)}^p \,d\mu(x) \\
           &\leq e^{pt\|V\|_{\infty}}\int_M\int_M|h(y)|^p_y\rho_t(x,y)\,d\mu(y)
           \bigg{(}\int_M\rho_t(x,z)\,d\mu(z) \bigg{)}^{p/q}\,d\mu(x) \\
           &\leq e^{pt\|V\|_{\infty}}\|h\|_{p}^p
\end{align*}
where $\rho_t(x,y)$ denotes the minimal heat kernel of $M$. It follows that $Q_t\colon L^p(E)\to L^p(E)$ are bounded operators for all $t \geq 0$.

As shown in the discussion following equation (17) in \cite{Guneysu-10}, the operator $Q_t$ satisfies
the equation
\begin{equation*}
(Q_t\psi)(x) = \psi(x) - \int_0^tQ_sH^{\nabla}_{p, V}\psi(x)ds
\end{equation*}
for all $\psi \in C_c^{\infty}(E)$.

Therefore $Q_t$ solves the following differential equation
\begin{equation}\nonumber
\frac{dQ_t}{dt}\psi = -Q_tH^{\nabla}_{p, V}\psi,\qquad Q_0\psi = \psi,
\end{equation}
for all $\psi \in C_c^{\infty}(E)$.

On the other hand, by lemma II.1.3 (ii) in \cite{engel-nagel}, the semi-group $S_t$ satisfies the same equation
\begin{equation}\nonumber
\frac{dS_t}{dt}\psi = -S_tH^{\nabla}_{p, V}\psi, \qquad
 S_0\psi = \psi,
\end{equation}
for all $\psi \in C_c^{\infty}(E)$. Hence, $Q_t\psi = S_t\psi$ for all $\psi \in C_c^{\infty}(E)$, and thus $Q_tf = S_tf$ for all $f \in L^p(E)$.
This proves the proposition in the case that $0\leq V\in C(\End E)\cap L^{\infty}(\End E)$.

Now assume $0\leq V \in L^{\infty}(\End E)$. By Lemma 3.1 of \cite{Guneysu-10}, we can find a sequence
$0\leq V_k \in  C(\End E)\cap L^{\infty}(\End E)$ such that for all $\psi \in C_c^{\infty}(E)$ we have
\begin{equation*}
\|H^{\nabla}_{p, V_k}\psi - H^{\nabla}_{p, V}\psi\|_p \rightarrow 0
\end{equation*}
as $k \rightarrow \infty$. Denote by $S_t^k$ the (strongly continuous, contractive) semigroup in $L^p(E)$ generated by $-H^{\nabla}_{p, V_k}$. As $C_c^{\infty}(E)$ is a common core for $H^{\nabla}_{p, V_k}$ and $H^{\nabla}_{p, V}$, it follows from the abstract Kato--Trotter
theorem, see Theorem III.4.8 in \cite{engel-nagel}, that $S^k_tf \rightarrow S_tf$ in $L^p(E)$, for all $f \in L^p(E)$, $1 < p < \infty$. From here, the proof of Theorem 1.1 in \cite{Guneysu-10} applies to obtain the formula \eqref{E:F-K} for
$0\leq V \in L^{\infty}(\End E)$.

The case of $0\leq V \in \lloc^{\infty}(\End E)$ proceeds in exactly the same way as Theorem 1.3 in \cite{Guneysu-10}.
\end{proof}

Before stating a corollary concerning the resolvent domination, we introduce the resolvent notations:
\begin{align*}
R^{\nabla}_V &:= (H^{\nabla}_{p,V} + 1)^{-1} : L^p(E) \rightarrow L^p(E), \\
R^{d}_{v} &:= (H^{d}_{p,v} + 1)^{-1} : L^p(M) \rightarrow L^p(M).
\end{align*}

With the formula~(\ref{E:F-K}) and the assumption $V\geq vI$ at our disposal, the proof of the following corollary is the same as that of property (iv) in Theorem 2.13 of~\cite{Guneysu-12}.
\begin{cor}\label{C:domination} Let $M$, $\nabla$, and $E$ be as in Proposition \ref{P:FK-rep}. Assume that $V\in\lloc^{\infty}(\End E)$ satisfies the inequality $V(x)\geq v(x)I$, for all $x\in M$, where $0\leq v\in\lloc^{\infty}(M)$. Then, for all $f \in L^p(E)$, $1<p<\infty$, we have
\begin{equation}\label{E:domination}
| R^{\nabla}_Vf(x) |_{E_{x}} \leq R^{d}_{v}  |f(x)|.
\end{equation}
\end{cor}

In the next proposition we state a coercive estimate for $L^{d}_{v}$. In the case $p=2$, assuming just geodesic completeness on $M$, the inequality~(\ref{M:Del_bound}) below was proven in Lemma 8 in \cite{Milatovic06-separation}. For the proof of~(\ref{M:Del_bound}) in the case $p\neq 2$ see Theorem 1.2 in~\cite{Milatovic13-separation}. Though stated under a bounded geometry hypothesis on $M$, the proof of the quoted result from~\cite{Milatovic13-separation}, which uses a sequence of second order cut-off functions along with $L^p$-positivity preservation property mentioned in section~\ref{S:intro-1} above,  works without change if we assume, in addition to geodesic completeness, that the Ricci curvature of $M$ is bounded from below by a constant. We should also mention that the two cited results from~\cite{Milatovic06-separation, Milatovic13-separation} use the assumption~\eqref{E:assumption-2}.

\begin{prop}\label{L:c-e} Let $M$ be as in the hypotheses of Theorem~\ref{T:main-2}. Assume that $0\leq v\in C^1(M)$ satisfies~(\ref{E:assumption-2}). Then, the following estimate holds for all $u\in D^{d}_p$:
\begin{equation}\label{M:Del_bound}
\|vu\|_{p} \leq C\|L^{d}_{v}u\|_{p} = C\|H^{d}_{p,v} u\|_{p},
\end{equation}
where $C\geq 0$ is a constant.
\end{prop}

\noindent \textbf{Continuation of the Proof of Theorem~\ref{T:main-2}} In the following discussion, $C$ will indicate a non-negative constant, not necessarily the same as the one in~(\ref{M:Del_bound}). Let $v \colon L^p(M) \rightarrow L^p(M)$ denote the maximal multiplication operator corresponding to the function $v$. We first show that the operator $vR^{d}_{v}: L^p(M) \rightarrow L^p(M)$ is bounded. Letting $w\in L^p(M)$ be arbitrary, we have $R^{d}_{v}w \in \Dom(H^{d}_{p,v})= D^{d}_p$. Applying \eqref{M:Del_bound} with $u=R^{d}_{v}w$, we obtain
\begin{align*}
&\|vR^{d}_{v}w\|_{p} \leq C\|H^{d}_{p,v}R^{d}_{v}w\|_{p} \leq C(\|w\|_p + \|R^{d}_{v}w\|_{p}) \leq C\|w\|_{p}.
\end{align*}
This proves $vR^{d}_{v}: L^p(M) \rightarrow L^p(M)$  is a bounded operator.
We then observe that by the boundedness of the operator $vR^{d}_{v}$, the assumption
$v(x)I \leq V(x) \leq \delta v(x)I$, and the
domination inequality~\eqref{E:domination}, we have
\begin{equation}\nonumber
\|V R^{\nabla}_Vf\|_{p} \leq \delta\|vR^{d}_{v}|f|\|_{p} \leq C\|f\|_{p},
\end{equation}
for all $f\in L^p(E)$. This shows that the operator $VR^{\nabla}_V\colon L^p(E)\rightarrow L^p(E)$ is bounded.

Let $h \in D^{\nabla}_p$ be arbitrary and write $Vh = VR^{\nabla}_V(h + L^{\nabla}_{V}h)$. Using the boundedness of the operator $V R^{\nabla}_V$, we obtain
\begin{equation}\nonumber
\|Vh\|_{p} \leq C(\|h\|_{p} + \|L^{\nabla}_{V}h\|_{p}),
\end{equation}
which shows that $L^{\nabla}_{V}$ is separated in $L^p(E)$. $\hfill\square$

%------------------------------------------------------------------------


\begin{thebibliography}{99}

\bibitem{AAR}
Atia, H.~A., Alsaedi, R.~S., Ramady, A.: Separation of bi-harmonic differential operators on Riemannian manifolds. Forum Math. \textbf{26} (2014) 953--966.

\bibitem{bs-16} Bandara, L., Saratchandran, H.:
Essential self-adjointness of powers of first-order differential operators on non-compact manifolds with low-regularity metrics.  J. Funct. Anal. \textbf{273} (2017) 3719--3758.

\bibitem{B-S-16} Bianchi, D., Setti,  A.~G.: Laplacian cut-offs, porous and fast diffusion on manifolds
and other applications. Calc. Var. \textbf{57}(4) (2018) https://doi.org/10.1007/s00526-017-1267-9

\bibitem{Boimatov88}
Boimatov, K.~Kh.: Coercive estimates and separation  for second order
elliptic differential equations, Soviet Math. Dokl. \textbf{38}(1) (1989)
157--160.

\bibitem{Boimatov97} Boimatov, K.~Kh.:
On the Everitt and Giertz method for Banach spaces,
Dokl. Akad. Nauk \textbf{356}(1) (1997) 10--12 (Russian).

\bibitem{bms}
Braverman, M., Milatovic, O., Shubin, M.: Essential self-adjointness of
Schr\"odinger type operators on manifolds, Russian Math. Surveys
57(4) (2002) 641--692.

\bibitem{engel-nagel} Engel, K.~J., Nagel, R.:
One-Parameter Semigroups for Linear Evolution Equations. Graduate Texts in Mathematics 194. Springer, Berlin (2000).

\bibitem{Everitt-Giertz77}
Everitt, W.~N.,  Giertz, M.: Inequalities and separation for
Schr\"odinger type operators in $L^2(\mathbb{R}^n)$, Proc. Royal
Soc. Edinburgh, \textbf{79A} (1977) 257--265.

\bibitem{GK}
Grummt, R., Kolb, M.: Essential selfadjointness of singular magnetic Schr\"odinger operators on Riemannian manifolds.
J. Math. Anal. Appl. \textbf{388} (2012) 480--489.

\bibitem{Guneysu-10}
G\"uneysu, B.: The Feynman--Kac formula for Schr\"odinger operators on
vector bundles over complete manifolds. J. Geom. Phys. \textbf{60} (2010) 1997--2010.

\bibitem{Guneysu-12}
G\"uneysu, B.: On generalized Schr\"odinger
semigroups. J. Funct. Analysis \textbf{262} (2012) 4639--4674.

%\bibitem{Guneysu-2014}
%G\"uneysu, B.: Sequences of Laplacian cut-off functions.  J. Geom. Anal.  \textbf{26} (2016) 171--184.

\bibitem{Guneysu-2016}
G\"uneysu, B.: Covariant Schr\"odinger Semigroups on Riemannian Manifolds.
Operator Theory: Advances and Applications 264. Birkh\"auser, Basel (2017).

\bibitem{Guneysu-Ulmer}
G\"uneysu, B.: The BMS-conjecture. Ulmer Seminare. Preprint: arXiv:1709.07463 (to appear).

\bibitem{GP-2017}
G\"uneysu, B., Pigola, S.: $L^p$-interpolation inequalities and global Sobolev regularity results. Annali di Matematica Pura ed Applicata. Preprint: arXiv:1706.00591v2 (to appear).

\bibitem{GP}
G\"uneysu, B., Post, O.: Path integrals and the essential self-adjointness of differential operators on noncompact manifolds.
Math. Z. \textbf{275} (2013) 331--348.


\bibitem{Kato80}
Kato, T.: Perturbation Theory for Linear Operators.
Springer-Verlag, New York (1980).



\bibitem{Milatovic06-separation}
Milatovic, O.: Separation property for Schr\"odinger operators on Riemannian manifolds,
J. Geom. Phys. \textbf{56} (2006) 1283--1293.

\bibitem{Mi-2010}
Milatovic, O.: On $m$-accretivity of perturbed Bochner Laplacian in $L^p$-spaces on Riemannian manifolds.
Integr. Equ. Oper. Theory \textbf{68} (2010) 243--254.


\bibitem{Milatovic13-separation}
Milatovic, O.: Separation property for Schr\"odinger operators $L^p$-spaces on non compact manifolds,
Complex Variables and Elliptic Equations \textbf{58} (2013) 853--864.

\bibitem{Milatovic18-separation}
Milatovic, O.: Self-Adjointness, m-accretivity, and separability for perturbations of Laplacian
and bi-Laplacian on Riemannian manifolds.
Integr. Equ. Oper. Theory (2018) 90:22. https://doi.org/10.1007/s00020-018-2452-8

\bibitem{hdn-12}
Nguyen, X.~D.: Essential selfadjointness and selfadjointness for even order elliptic operators.
Proc. Roy. Soc. Edinburgh Sect. A \textbf{93} (1982/83) 161--179.

\bibitem{Okazawa-84} Okazawa, N.:
An $L^{p}$ theory for Schr\"odinger operators with nonnegative potentials. J.~Math.~Soc.~Japan
\textbf{36} (1984) 675–-688.

\bibitem{Pran-Ser-Riz-16} Prandi, D., Rizzi, L., Seri, M.: Quantum confinement on non-complete Riemannian manifolds.
Journal of Spectral Theory. Preprint: arxiv:1609.01724 (to appear).

\bibitem{Strichartz-83}
Strichartz, R.: Analysis of the Laplacian on the complete Riemannian manifold. J. Funct. Anal. \textbf{52}(1)
(1983) 48--79.



\end{thebibliography}
\end{document}